\documentclass[11pt,a4paper]{amsart}
\usepackage{graphicx,multirow,array,amsmath,amssymb,upgreek,color,mathtools,extarrows}
\usepackage[colorlinks=true,linkcolor=red, citecolor=blue]{hyperref}
 
\usepackage[all]{xy}
\newtheorem{theorem}{Theorem}
\newtheorem{proposition}[theorem]{Proposition}
\newtheorem{lemma}[theorem]{Lemma}
\newtheorem{cor}[theorem]{Corollary}
\newtheorem*{definition}{Definition}

\usepackage{multirow}
\usepackage[table,xcdraw]{xcolor}

\newcommand{\veq}{\mathrel{\rotatebox{270}{$\simeq$}}}
\def\stackbelow#1#2{\underset{\displaystyle\overset{\displaystyle\veq}{#2}}{#1}}

\begin{document}

\title{Chirality for simple graphs of size up to 12}
\author[H. Choi]{Howon Choi}
\address{Institute of Natural Science, Korea University, 2511 Sejong-ro, Sejong City, 30019, Korea}
\email{howon@korea.ac.kr}
\author[H. Kim]{Hyoungjun Kim}
\address{College of General Education, Kookmin University, Seoul 02707, Korea}
\email{kimhjun@kookmin.ac.kr}
\author[S. No]{Sungjong No}
\address{Department of Mathematics, Kyonggi University, Suwon 16227, Korea}
\email{sungjongno@kgu.ac.kr}

\begin{abstract}
Chirality is one of the important assymmetrical property in wide area of natural science, which has been studied to predict molecular behavior.
One of good methods to analyze molecules with complex structures is representing them as graphs embedded in 3-dimensional space.
So it is important to study the chirality of spatial graphs to understand structure of chiral molecules.
Moreover, Robertson and Seymour's graph minor theorem implies that a set of minor minimal graphs with respect to intrinsic properties is finite.
So it is also important to find a complete set of minor minimal graphs for intrinsic properties.
In this paper, we classify minor minimal intrinsically chiral graphs among simple graphs of size up to twelve.
\end{abstract}


\maketitle

\section{Introduction}\label{sec:int}

Chirality is one of the important assymmetrical property in mathematics, physics, chemistry, biology, pharmacology, etc.
In particular, chirality has been studied to predict molecular behavior.
In chemistry, a molecule is said to be {\it chiral\/} if it cannot convert itselt into its mirror image, otherwise it is said to be {\it achiral\/}.
Since chiral molecules are distinguishable from their mirror images, each of them has two types of molecular structures.
Such two types of the same chiral molecule, called {\it stereoisomers\/}, are molecules that have the same molecular formula and the same bonding arrangement, but differ in the arrangement of their atoms in the space.
{\it Enantiomers\/} are pairs of stereoisomers which are mirror images of each other.
Even though enantiomers have the same physical properties such as density, melting points and boiling points, they may have the different chemical properties.
For example, pharmaceuticals using chiral compound, such as thalidomide, ethambutol, penicillin, and naproxen, may have no-effect or side effects in enantiomers.
Therefore, such enantiomers have been actively researched not only in pharmacology but also in wide area of science~\cite{DA,MB,G,JAM,PB}

One of good methods to analyze molecules with complex structures is representing them as graphs embedded in 3-dimensional space which are called {\it spatial graph\/}.
Most of results in the spatial graph theory have their roots on the result of Conway and Gordon~\cite{CG}.
They showed that every embedding of a complete graph $K_6$ contains a non-splittable link and every embedding of a complete graph $K_7$ contains a non-trivial knot.
These properties for $K_6$ and $K_7$ are called {\it intrinsically linked\/} and {\it intrinsically knotted\/}, respectively.
A graph $H$ is a {\it minor\/} of another graph $G$ if $H$ can be obtained from $G$ by contracting edges, deleting edges, and deleting isolated vertices.
If a graph $G$ with intrinsically linked has no proper minor that is the intrinsically linked, then we say $G$ is {\it minor minimal intrinsically linked\/}.
Similarly, we also define {\it minor minimal intrinsically knotted\/}.

\begin{figure}[h!]
    \includegraphics{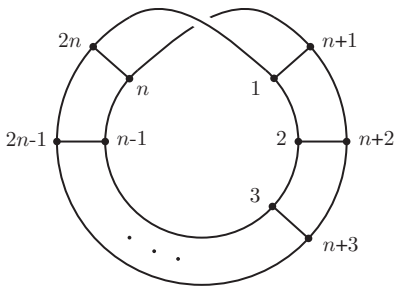}
    \caption{The standard embedding of $M_n$}
    \label{fig:mn}
\end{figure}

A molecule is a group of two or more atoms held together by chemical bonds.
So we study a spatial graph to analyze chirality of molecules with complex structure, since vertices and edges in a spatial graph correspond to atoms and chemical bonds in a molecule, respectively.
An embedding of a graph $G$ is {\it topologically chiral\/} if it cannot be ambient isotopic to its mirror image, otherwise it is {\it topologically achiral\/}.
A graph $G$ is {\it intrinsically chiral\/} if every embedding of $G$ in 3-dimensional space is topologically chiral, otherwise it is {\it achirally embeddable\/}.
Guy and Harary~\cite{GH} first introduced graphs {\it M{\"o}bius ladders\/} $M_n$ which consists of a polygon of length $2n$ and all $n$ chords joining opposite pairs of vertices where $n\ge3$.
Note that for a M{\"o}bius ladder with $n$ chords, they originally defined as $M_{2n}$, but since the M{\"o}bius ladder has been used as $M_n$ in many results in spatial graph theory.
So we follow the notation $M_n$ for a M{\"o}bius ladder with $n$ chords.
We call a polygon of length $2n$ as a {\it loop\/}, denoted by $K$, and the chords as {\it rungs\/}, denoted by $\alpha_1, \dots, \alpha_n$.
Especially, we say $M_n$ is {\it standardly embedded\/} in $S^3$ when $M_n$ is embedded as drawn in Figure~\ref{fig:mn} or its mirror image.
Simon~\cite{S} showed that for every standard embedding of $M_n$, there is no orientation reversing diffeomorphism $h$ of $S^3$ which satisfies $h(M_n)=M_n$ and $h(K)=K$ when $n \ge 3$.
Flapan~\cite{F} showed that every embedding of $M_n$, there is no orientation reversing diffeomorphism $h$ of $S^3$ with $h(M_n)=M_n$ and $h(K)=K$ when $n \ge 3$ and $n$ is odd.
From this result, Flapan and Weaver~\cite{FW1} showed that for positive integer $n$. the complete graphs $K_{4n+3}$ are intrinsically chiral, and all other complete graphs are achirally embeddable.
Flapan and Fletcher~\cite{FF} classified intrinsic chirality for complete multipartite graphs.
Flapan and Weaver~\cite{FW2} defined three types for an automorphism of a graph, and showed that the graph is achirally embeddable if and only if there exists an automorphism of the graph among the three types.

\begin{figure}[h!]
    \includegraphics{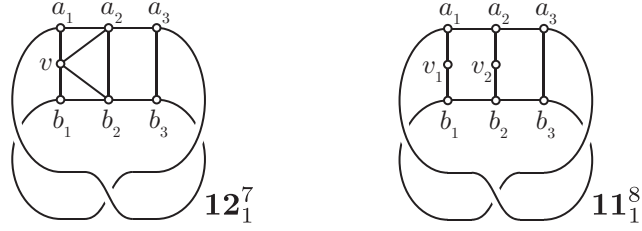}
    \caption{Minor minimal intrinsically chiral graphs $\mathbf{12}^7_1$ and $\mathbf{11}^8_1$}
    \label{fig:g78}
\end{figure}

A graph property $\mathcal{P}$ is {\it minor closed\/} if every minor of a $\mathcal{P}$ graph is also $\mathcal{P}$.
Note that if $\mathcal{P}$ is planarity, linklessly embeddability and knotlessly embeddability, then $\neg \mathcal{P}$ is non-planarity, intrinsic linkedness and intrinsic knottedness, respectively.
Moreover, these three properties are minor closed.
It is known that Robertson and Seymour's~\cite{RS} graph minor theorem implies that the set of minor minimal $\neg \mathcal{P}$ graphs is finite when $\mathcal{P}$ is minor closed.
Complete sets for minor minimal non-planar graphs~\cite{K,W} and intrinsically linked graphs~\cite{RST} were found.
Recently, classfying graphs with minor minimal intrinsically knotted or its related properties have been actively researched~\cite{BM,GMN,KMO,LKLO}.
Flapan and Weaver's result~\cite{FW1} implies that an achirally embeddability is not minor closed.
Mattman~\cite{M} showed that the graph minor theorem also implies that the set of minor minimal $\neg \mathcal{P}$ graphs is finite for any graph property $\mathcal{P}$.
This means that there are finite number of minor minimal intrinsically chiral graphs.
Our overall goal is finding the complete set of minor minimal intrinsically chiral graphs.
Authors~\cite{CKN} classified all simple intrinsically chiral graphs of size up to eleven, and found two minor minimal intrinsically chiral graphs $\varGamma_7$ and $\varGamma_8$ as drawn in Figure~\ref{fig:g78}.
In here, the {\it size\/} of a graph is its number of edges.
Similarly, the {\it order\/} of a graph is its number of vertices.
Henceforth, we use notations $\mathbf{12}^7_1$ and $\mathbf{11}^8_1$ instead of $\varGamma_7$ and $\varGamma_8$, respectively.
Note that $\mathbf{11}^8_1$ means the first simple graph among minor minimal graphs of size eleven and order eight in descending order of the degree sequence of vertices.
In this paper, we present four new simple minor minimal intrinsicallay chiral graphs $\mathbf{12}^7_2$, $\mathbf{12}^8_1$, $\mathbf{12}^8_2$ and  $\mathbf{12}^9_1$ as drawn in Figure~\ref{fig:ic12}.
We also classify all simple minor minimal intrinsically chiral graphs of size twelve.

\begin{figure}[h!]
    \includegraphics{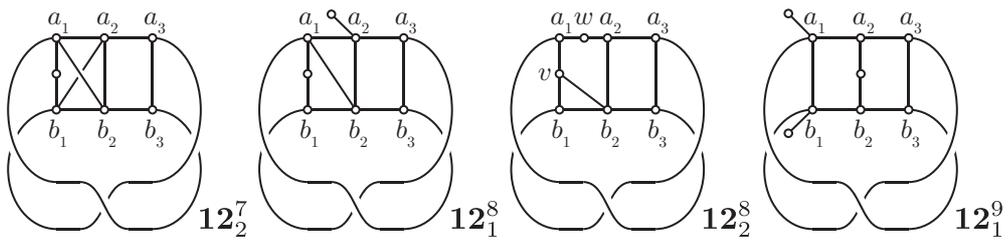}
    \caption{New minor minimal intrinsically chiral graphs}
    \label{fig:ic12}
\end{figure}

\begin{theorem}
There are exactly five simple minor minimal intrinsically chiral graphs of size twelve.
\end{theorem}

These five graphs are $\mathbf{12}^7_1$ and the four new graphs in Figure~\ref{fig:ic12}.
We remark that $\mathbf{11}^8_1$ is the unique simple minor minimal intrinsically chiral graph of size up to eleven.
Thus we have the following corollary.

\begin{cor}
There are exactly six simple minor minimal intrinsically chiral graphs of size up to twelve.
\end{cor}

\section{Terminology}\label{sec:term}

A {\it loop\/} is an edge which connects a vertex to itself.
If there are more than one edge which connect the same pair of vertices, then we say these edges are {\it multiple edges}.
Especially, if multiple edges consist of exactly two edges, then we say {\it double edges\/}.
A {\it simple graph} is a graph without loops and multiple edges.

From now on, let $G=(V,E)$ denote a simple graph with a set of vertices $V$ and a set of edges $E$.
For notational convenience, we introduce the following notations.
\begin{itemize}
    \item $|G|$ is the order of $G$.
    \item $\|G\|$ is the size of $G$.
    \item $e_{vw}$ is an edge which connects vertices $v$ and $w$.
    \item $G/ e$ is a graph obtained from $G$ by contracting an edge $e$.
    \item $G\setminus e$ is a graph obtained from $G$ by deleting an edge $e$.
    \item $G\setminus \{v\}$ is a graph obtained from $G$ by deleting a vertex $v$ and its incident edges.
    \item $G \simeq G'$ denotes that $G$ is isomorphic to $G'$.
\end{itemize}
Note that if $G$ is intrinsically chiral with more than one component, then there exists at least one component which is intrinsically chiral.
This implies that $G$ is not minor minimal intrinsically chiral when $G$ is not connected.
So we further assume that $G$ is connected.

We first consider the achirally embeddability of the given graph.
Flapan and Weaver~\cite{FW2} found three conditions, called automorphism of a graph of types, which guarantee the given graph is achirally embeddable.
Among them, we use the first type as following.
\begin{definition}{\cite{FW2}}
    \normalfont
Let $G$ be a graph with an order 2 automorphism $\phi$.
We say $\phi$ is of {\it type} 1 if the vertices of $G$ can be partitioned into three sets, $V_1$, $W_2$, and $W'_2$ which satisfy:
\begin{enumerate}
    \item $\phi$ fixes $V_1$ and $\phi$ interchanges the vertices of $W_2$ and $W'_2$.
    \item The subgraph induced by $V_1$ is planar.
    \item If there is an edge with vertex $p$ in $W_2$ and vertex $q$ in $W'_2$, then $\phi(p)=q$.
\end{enumerate}
\end{definition}
They showed that if $G$ has an automorphism of type 1, then $G$ is achirally embeddable.
A {\it mirror symmetrical embedding} is an embedding of a graph which is symmetrical on the left and right with respect to a plane as drawn in Figure~\ref{fig:mirror}.
We call $\mathcal{M}$ a {\it mirror plane}.

\begin{figure}[h!]
    \includegraphics[scale=1.3]{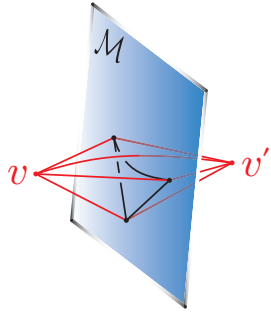}
    \caption{A mirror symmetrical embedding $M(G)^{v}_{v'}$}
    \label{fig:mirror}
\end{figure}

Note that if a graph has a mirror symmetrical embedding, then it has an automorphism of type 1.
This means that the graph is achirally embeddable.
So we give the following two lemmas which are frequently used in the rest of the paper.

\begin{lemma}\label{lem:ae}
    Let $G$ be a simple graph such that two vertices $v$ and $v'$ are adjacent to the same vertices in $G \setminus \{ v, v' \}$. 
    If $G \setminus \{ v, v' \}$ is planar, then $G$ is achirally embeddable.
\end{lemma}

Especially, a graph which satisfies the assumption of Lemma~\ref{lem:ae} has a mirror symmetrical embedding such that each of the left and the right side of the mirror plane only contains exactly one vertex $v$ and $v'$, respectively.
We denote $M(G)^{v}_{v'}$ for a such mirror symmetrical embedding.

\begin{lemma}\label{lem:c4}
Let $G$ be a simple graph which has two mirror symmetrical embeddings $M(G)^{v_1}_{v_2}$ and $M(G)^{v_3}_{v_4}$ such that $v_1$ and $v_2$ are adjacent to $v_3$ and $v_4$.
Further assume that a simple graph $G'$ obtained from $G$ by adding an edge.
If both $G'\backslash\{v_1,v_2\}$ and $G'\backslash\{v_3,v_4\}$ are planar, then $G'$ is achirally embeddable.
\end{lemma}

\begin{proof}
Let $G$ and $G'$ be graphs satisfying the assumption, respectively.
Further assume that $e$ is the added edge to obtain $G'$ from $G$.
Suppose for the contradiction that $G'$ is intrinsically chiral.
If $e$ is not adjacent to any vertex in $\{v_1,v_2\}$, then $G'$ also has a mirror symmetrical embedding $M(G')^{v_1}_{v_2}$.
So $e$ is adjacent to at least one vertex in $\{v_1,v_2\}$.
Without loss of generality, we assume that $e$ is connected to $v_1$.
Since $G'$ is simple, $e$ cannot be connected to $v_3$ or $v_4$.
This implies that $G'$ has a mirror symmetrical embedding $M(G')^{v_3}_{v_4}$, a contradiction.
\end{proof}

Now we consider the intrinsic chirality of a graph.
To determine whether the given graph is intrinsically chiral, we use the following proposition which is shown by Flapan~\cite{F}.

\begin{proposition}\label{prop:fla}
Let $M_n$ be a M{\"o}bius ladder which is embedded in $S^3$ with loop $K$, where $n$ is an odd number.
Then there is no diffeomorphism $h:S^3 \rightarrow S^3$ which is orientation reversing with $h(M_n)=M_n$ and $h(K)=K$.
\end{proposition}

Note that we only deal with intrinsic chirality of a simple connected graph of size twelve in the main theorem.
The authors~\cite{CKN} showed that there does not exist any simple connected intrinsically chiral graph of order up to six.
Furthermore, they also introduce the following lemma.

\begin{lemma}{\cite{CKN}} \label{lem:pla}
    Let $G$ be a connected graph.
    If $\|G\|-|G| \le 2$, then $G$ is planar.
\end{lemma}    

By this lemma, we only need to consider the cases that the order of $G$ is seven, eight or nine, since $G$ is planar when $|G| \ge 10$.
Since the only simple connected intrinsically chiral graph of size up to eleven is $\mathbf{11}^8_1$~\cite{CKN}, we may assume that $G$ does not contain $\mathbf{11}^8_1$ as a minor.

We first show that the four graphs $\mathbf{12}^7 _2$, $\mathbf{12}^8 _1$, $\mathbf{12}^8 _2$ and $\mathbf{12}^9 _1$ are minor minimal intrinsically chiral in Section~\ref{sec:mmic},
and then we classify the simple minor minimal intrinsically chiral graphs of order seven, eight and nine in Section~\ref{sec:7v}, \ref{sec:8v} and \ref{sec:9v}, respectively.

\section{The minor minimal intrinsically chiral graphs}\label{sec:mmic}

In this section, we show that four graphs, $\mathbf{12}^7_2$, $\mathbf{12}^8_1$, $\mathbf{12}^8_2$ and $\mathbf{12}^9_1$, are minor minimal intrinsically chiral.
Even though these four graphs are simple, not all minors of the four graphs are simple.
Hence, we should also check the intrinsic chirality of non-simple minors included in the four graphs.
To conveniently determine the intrinsic chirality for all minors of a graph $G$, we use the following diagram, called {\it set-wise commutative ladder}, which consists of sets of minors of $G$.

{\small{ $$\xymatrix{
 \ \ \ \ (|G|,\|G\|) \ \ \ \  \ar[r]^-{ \mathbb {D}} \ar[d]_{ \mathbb{C}} & \ \ (|G|,\|G\|-1) \ \ \ar[r] \ar[d] & \ \ (|G|,\|G\|-2) \ \ \ar[r] \ar[d] & \cdots\\
 (|G|-1,\|G\|-1) \ar[r] \ar[d] & (|G|-1,\|G\|-2) \ar[r] \ar[d] & (|G|-1,\|G\|-3)\ar[r]\ar[d]   & \cdots \\
 (|G|-2,\|G\|-2)\ar[r] \ar[d] &(|G|-2,\|G\|-3) \ar[r] \ar[d]& (|G|-2,\|G\|-4)\ar[r]\ar[d] & \cdots \\
 \vdots  & \vdots & \vdots & \\
 },$$ }}
where $\mathbb{D}  : G \to G\setminus e$ and $\mathbb{C} :G \to G/e $ for any edge $e$ of $G$.

Note that the graph obtained from a simple graph by an edge deletion is simple, and the graph obtained from a connected graph by an edge contraction is connected.

\begin{lemma}\label{lem:rde}
Let $G$ be a simple graph of size twelve which does not contain $\mathbf{11}^8_1$ as a minor.
If $G$ is intrinsically chiral, then it is minor minimal.
\end{lemma}

\begin{proof}
Let $G$ be a simple intrinsically chiral graph of size twelve which does not contain $\mathbf{11}^8_1$ as a minor.
Since $\mathbf{11}^8_1$ is the unique simple intrinsically chiral graph of size up to eleven, it sufficient to show that every non-simple minor of $G$ is achirally embeddable.
Let $G'$ be a minor of $G$ which is non-simple and non-planar.
Now we consider the set-wise commutative ladder of $G$ as following,
$${\small{ \xymatrix{
 \ \ \ \ (|G|,12) \ \ \ \  \ar[r]^-{ \mathbb {D}} \ar[d]_{ \mathbb{C}} & \ \ (|G|,11) \ \ \ar[r] \ar[d] & \ \ (|G|,10) \ \ \ar[r] \ar[d] & \cdots\\
 (|G|-1,11) \ar[r] \ar[d] & (|G|-1,10) \ar[r] \ar[d] & (|G|-1,9)\ar[r]\ar[d]   & \cdots \\
 (|G|-2,10)\ar[r] \ar[d] &(|G|-2,9) \ar[r] \ar[d]& (|G|-2,8)\ar[r]\ar[d] & \cdots \\
 \vdots  & \vdots & \vdots & \\
 }.}}$$
We remark that every graph in the first row is simple, since an edge deletion preserves simplicity of a graph.
Moreover the size of a graph cannot exceed nine, when it is obtained from $G$ by using edge contractions more than twice.
Since a non-simple graph of size up to nine is planar, we have two cases that $G'$ is in the second or the third row.

First assume that $G'$ is in the second row.
Then $G'$ has double edges, since every graph in the second row is obtained from simple graphs by an edge contraction once.
Let $e_1$ and $e_2$ be double edges of $G'$.
Since the size of $G'$ is at most eleven, $G' \setminus (e_1 \cup e_2)$ is either planar or $K_{3,3}$
If $G' \setminus (e_1 \cup e_2)$ is planar, then $G'$ is achirally embeddable.
In detail, first put $G' \setminus (e_1 \cup e_2)$ on the mirror plane.
Next, put $e_1$ and $e_2$ to be symmetrically embedded on the left and right with respect to the mirror plane.
Thus $G' \setminus (e_1 \cup e_2)$ is $K_{3,3}$(or may have an isolated vertex).
Note that every pair of non-adjacent vertices of $G' \setminus (e_1 \cup e_2)$ can be lied on the mirror plane.
Similar with the above case, we have a mirror symmetrical embedding of $G'$.
So $G'$ is achirally embeddable.

Now assume that $G'$ is in the third row.
The size of every graph in the third row is at most ten.
This implies that $G'$ has $K_{3,3}$ as a minor.
So there are two cases that $G'$ has either double edges or a loop.
If $G'$ has double edges, then the graph obtained from $G'$ by deleting double edges is planar.
If $G'$ has a loop, then there is a mirror symmetrical embedding of $K_{3,3}$ such that the vertex connected to the loop is lied on the mirror plane.
Therefore $G'$ is achirally embeddable in both cases.
\end{proof}

We remark that four graphs $\mathbf{12}^7_2$, $\mathbf{12}^8_1$, $\mathbf{12}^8_2$ and $\mathbf{12}^9_1$ do not contain $\mathbf{11}^8_1$ as a minor.
Thus it is sufficient to show that these four graphs are intrinsically chiral by Lemma~\ref{lem:rde}.

\subsection{$\mathbf{12}^7_2$ is minor minimal intrinsically chiral}\label{subsec:72}\

Let $h:S^3 \rightarrow S^3$ be a homeomorphism such that $h(\mathbf{12}^7_2) = \mathbf{12}^7_2$.
The homeomorphism $h$ induces an automorphism on the vertices of $\mathbf{12}^7_2$.
Note that $\mathbf{12}^7_2$ consists of four degree 4 vertices $a_1$, $a_2$, $b_1$ and $b_2$, two degree 3 vertices $a_3$ and $b_3$, and one degree 2 vertex $v$.
More precisely, $a_1$ and $b_1$ are adjacent to $v$, but $a_2$ and $b_2$ are not adjacent to $v$.
So the connection of a vertex in $\{ a_1, b_1 \}$ is different from a vertex in $\{ a_2, b_2 \}$.
Thus we have that 
\begin{eqnarray*}
  \{a_1,~b_1\} \xlongrightarrow{h} \{a_1,~b_1\}\\
  \{a_2,~b_2\} \xlongrightarrow{h} \{a_2,~b_2\}\\
  \{a_3,~b_3\} \xlongrightarrow{h} \{a_3,~b_3\}\\
  \{ v \} \xlongrightarrow{h} \{ v \} \qquad 
\end{eqnarray*}
Let $K=a_1a_2a_3b_1b_2b_3a_1$ be a cycle of $\mathbf{12}^7_2$.
Then we have
\begin{eqnarray*}
h(a_1a_2a_3b_1b_2b_3a_1) =
\begin{cases}
 a_1a_2a_3b_1b_2b_3a_1 & \text{if~}h(a_1)=a_1,\\
 b_1b_2b_3a_1a_2a_3b_1 & \text{if~}h(a_1)=b_1.
\end{cases}
\end{eqnarray*}
Thus $h(K)=K$ in $\mathbf{12}^7_2$.

Now consider a graph $M_3$ which is obtained from $\mathbf{12}^7_2$ by deleting two edges $e_{a_1b_2}$ and $e_{b_1a_2}$ and contracting an edge $e_{va_1}$.
Since a homeomorphism $h$ satisfies $h(M_3)=M_3$ and $h(K)=K$, $h$ cannot be an orientation reversing homeomorphism by Proposition~\ref{prop:fla}.
Therefore, $\mathbf{12}^7_2$ is intrinsically chiral.

\subsection{$\mathbf{12}^8_1$ is minor minimal intrinsically chiral}\label{subsec:81}\

Let $h:S^3 \rightarrow S^3$ be a homeomorphism such that $h(\mathbf{12}^8_1) = \mathbf{12}^8_1$.
The homeomorphism $h$ induces an automorphism on the vertices of $\mathbf{12}^8_1$.
Note that $\mathbf{12}^8_1$ consists of two degree 4 vertices $a_2$ and $b_2$, four degree 3 vertices $a_1$, $a_3$, $b_1$ and $b_3$, one degree 2 vertex $v$, and one degree 1 vertex $w$.
More precisely, $a_2$ is adjacent to $w$, but $b_2$ is not adjacent to $w$.
So the connection of $a_2$ is different from a vertex of $b_2$.
Similarly, $a_1$ and $b_1$ are adjacent to $v$, but $a_2$ and $b_2$ are not adjacent to $v$.
So the connection of a vertex in $\{ a_1, b_1 \}$ is different from a vertex in $\{ a_2, b_2 \}$.
Thus $h$ induces identity automorphism.
This implies that $h(K)=K$ in $\mathbf{12}^8_1$.

Now consider a graph $M_3$ which is obtained from $\mathbf{12}^8_1$ by deleting an edge $e_{a_1b_2}$, and contracting two edges $e_{va_1}$ and $e_{wa_2}$.
Since a homeomorphism $h$ satisfies $h(M_3)=M_3$ and $h(K)=K$, $h$ cannot be an orientation reversing homeomorphism by Proposition~\ref{prop:fla}.
Therefore, $\mathbf{12}^8_1$ is intrinsically chiral.

\subsection{$\mathbf{12}^8_2$ is minor minimal intrinsically chiral}\label{subsec:82}\

Let $h:S^3 \rightarrow S^3$ be a homeomorphism such that $h(\mathbf{12}^8_2) = \mathbf{12}^8_2$.
The homeomorphism $h$ induces an automorphism on the vertices of $\mathbf{12}^8_2$.
Note that $\mathbf{12}^8_2$ consists of a degree 4 vertex $b_2$, six degree 3 vertices $v$, $a_1$, $a_2$, $a_3$, $b_1$ and $b_3$, and one degree 2 vertex $w$.
More precisely, $a_2$ the only degree 3 vertex that is adjacent to both a degree 4 vertex $b_2$ and a degree 2 vertex $w$.
$a_1$ is the only vertex which is adjacent to $w$ among $\{v, a_1, a_3, b_1, b_3 \}$
Similarly, $a_3$ is the only vertex which is adjacent to $a_2$ among $\{v, a_3, b_1, b_3 \}$, and $b_3$ is the only vertex which is adjacent to both $a_1$ and $a_3$ among $\{v, b_1, b_3 \}$.
Finally the connection of $v$ is different from a vertex of $b_1$.
Thus $h$ induces identity automorphism.
This implies that $h(K)=K$ in $\mathbf{12}^8_2$.

Now consider a graph $M_3$ which is obtained from $\mathbf{12}^8_1$ by deleting an edge $e_{vb_2}$, and contracting two edges $e_{va_1}$ and $e_{wa_1}$.
Since a homeomorphism $h$ satisfies $h(M_3)=M_3$ and $h(K)=K$, $h$ cannot be an orientation reversing homeomorphism by Proposition~\ref{prop:fla}.
Therefore, $\mathbf{12}^8_2$ is intrinsically chiral.

\subsection{$\mathbf{12}^9_1$ is minor minimal intrinsically chiral}\label{subsec:91}\

Let $h:S^3 \rightarrow S^3$ be a homeomorphism such that $h(\mathbf{12}^9_1) = \mathbf{12}^9_1$.
The homeomorphism $h$ induces an automorphism on the vertices of $\mathbf{12}^9_1$.
Note that $\mathbf{12}^9_1$ consists of two degree 4 vertices $a_1$ and $b_1$, four degree 3 vertices $a_2$, $a_3$, $b_2$ and $b_3$, and one degree 2 vertex $v$, and two degree 1 vertex $w_1$ and $w_2$.
More precisely, $a_2$ and $b_2$ are adjacent to $v$, but $a_3$ and $b_3$ are not adjacent to $v$.
So the connection of a vertex in $\{ a_2, b_2 \}$ is different from a vertex in $\{ a_3, b_3 \}$.
Thus we have that 
\begin{eqnarray*}
  \{w_1,~w_2\} \xlongrightarrow{h} \{w_1,~w_2\}\\
  \{a_1,~b_1\} \xlongrightarrow{h} \{a_1,~b_1\}\\
  \{a_2,~b_2\} \xlongrightarrow{h} \{a_2,~b_2\}\\
  \{a_3,~b_3\} \xlongrightarrow{h} \{a_3,~b_3\}\\
  \{ v \} \xlongrightarrow{h} \{ v \} \qquad 
\end{eqnarray*}
Let $K=a_1a_2a_3b_1b_2b_3a_1$ be a cycle of $\mathbf{12}^9_1$.
Then we have
\begin{eqnarray*}
h(a_1a_2a_3b_1b_2b_3a_1) =
\begin{cases}
 a_1a_2a_3b_1b_2b_3a_1 & \text{if~}h(a_1)=a_1,\\
 b_1b_2b_3a_1a_2a_3b_1 & \text{if~}h(a_1)=b_1.
\end{cases}
\end{eqnarray*}
Thus $h(K)=K$ in $\mathbf{12}^9_1$.

Now consider a graph $M_3$ which is obtained from $\mathbf{12}^9_1$ by contracting three edges $e_{w_1a_1}$, $e_{w_2b_1}$ and $e_{va_2}$.
Since a homeomorphism $h$ satisfies $h(M_3)=M_3$ and $h(K)=K$, $h$ cannot be an orientation reversing homeomorphism by Proposition~\ref{prop:fla}.
Therefore, $\mathbf{12}^9_1$ is intrinsically chiral.

\section{The order of $G$ is seven}\label{sec:7v}

In this section, we assume that $G$ is a non-planar graph of order seven and size twelve.
So $G$ has $K_{3,3}$ or $K_5$ as a minor.
First suppose that $G$ has $K_{3,3}$ as a minor.
Then there exist three edges $e_1$, $e_2$ and $e_3$ such that $K_{3,3}$ is obtained from $G$ by using two edge deletions at $e_1$ and $e_2$, and one edge contraction at $e_3$.
$$\xymatrix{
\ \ G \ \ \ar[r]^-{ \mathbb {D}} & \ \ G \setminus e_1 \ \ \ar[r]^-{ \mathbb {D}} & \ \ (G \setminus e_1) \setminus e_2 \ \ \ar[r]^-{ \mathbb {C}} & \ \ \stackbelow{((G \setminus e_1) \setminus e_2) /e_3}{K_{3,3}}\ \ \\
 }$$
In the process to obtain $K_{3,3}$ from $G$, we use all edge deletions before edge contractions.
Since simplicity is preserved by edge deletions, all four graphs in the process are simple.
Note that there are exactly two graphs $H_1$ and $H_2$ which can obtain $K_{3,3}$ by exactly one edge contraction as drawn in Figure~\ref{fig:k33ec}.
In here, $H_1$ has a degree 1 vertex and $H_2$ has a degree 2 vertex.

\begin{figure}[h!]
    \includegraphics{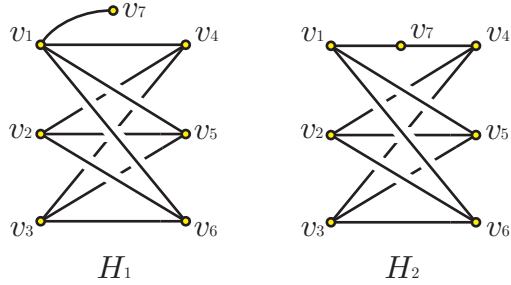}
    \caption{Two graphs $H_1$ and $H_2$ which can obtain $K_{3,3}$ by exactly one edge contraction}
    \label{fig:k33ec}
\end{figure}

Now we have two cases that $(G \setminus e_1) \setminus e_2$ is isomorphic to $H_1$ and $H_2$.
We prove the first and the second cases in Subsection~\ref{subsec:7v1} and \ref{subsec:7v2}, respectively.
In addition, we deal with the case that $G$ has $K_5$ as a minor in Subsection~\ref{subsec:7v3}.

\subsection{Case of $(G \setminus e_1) \setminus e_2 \simeq H_1$}\label{subsec:7v1}\

First suppose that $e_1$ is not connected to $v_7$.
Assume that $e_1$ is not connected to any vertex in $\{v_4,v_5,v_6\}$.
There are two cases of the connection of $e_2$ such that it is connected at most one vertex in $\{v_4,v_5,v_6\}$ or it connects two of them.
In the first case, we assume that $e_2$ is not connected to $v_4$ and $v_5$, and in the second case, we assume that $e_2$ connects $v_4$ and $v_5$ as drawn in Figure~\ref{fig:v7} (a).
Since both $v_4$ and $v_5$ are only adjacent to $v_1$, $v_2$ and $v_3$ in a planar graph $G \setminus \{ v_4, v_5 \}$, $G$ is achirally embeddable by Lemma~\ref{lem:ae}, in both cases.
Thus $e_1$ is connected to at least one vertex in $\{v_4,v_5,v_6\}$.
Without loss of generality, we assume that $e_1$ connects $v_4$ and $v_5$, since $G \setminus e_2$ is simple as drawn in Figure~\ref{fig:v7} (b).
Then $G \setminus e_2$ has two mirror symmetrical embeddings $M(G \setminus e_2)^{v_2}_{v_3}$ and $M(G \setminus e_2)^{v_5}_{v_6}$.
By Lemma~\ref{lem:c4}, $G$ is achirally embeddable.
Thus both $e_1$ and $e_2$ are connected to $v_7$.

Now suppose that $e_1$ is not connected to any vertex in $\{v_4,v_5,v_6\}$.
Then at least two vertices in $\{v_4,v_5,v_6\}$ which are not connected to $e_2$, say $v_4$ and $v_5$.
Since both $v_4$ and $v_5$ are only adjacent to $v_1$, $v_2$ and $v_3$ in a planar graph $G \setminus \{ v_4, v_5 \}$, $G$ is achirally embeddable by Lemma~\ref{lem:ae}.
So we assume that $e_1$ is connected to $v_4$ as drawn in Figure~\ref{fig:v7} (c).
Then $G \setminus e_2$ has two mirror symmetrical embeddings $M(G \setminus e_2)^{v_2}_{v_3}$ and $M(G \setminus e_2)^{v_5}_{v_6}$.
By Lemma~\ref{lem:c4}, $G$ is achirally embeddable.

\subsection{Case of $(G \setminus e_1) \setminus e_2 \simeq H_2$}\label{subsec:7v2}\

If both $e_1$ and $e_2$ are not connected to $v_2$ and $v_3$, then both $v_2$ and $v_3$ are only adjacent to $v_4$, $v_5$ and $v_6$ in a planar graph $G \setminus \{ v_2, v_3 \}$.
Thus $G$ is achirally embeddable by Lemma~\ref{lem:ae}.
So at least one of $e_1$ and $e_2$ is connected to $v_2$ or $v_3$.
Similarly, at least one of $e_1$ and $e_2$ is connected to $v_5$ or $v_6$.
Without loss of generality, we assume that $e_1$ is connected to $v_2$, and $e_2$ is connected to $v_5$.
If $e_1$ is connected to $v_3$ as drawn in Figure~\ref{fig:v7} (d), then $G \setminus e_2$ has two mirror symmetrical embeddings $M(G \setminus e_2)^{v_2}_{v_3}$ and $M(G \setminus e_2)^{v_5}_{v_6}$.
By Lemma~\ref{lem:c4}, $G$ is achirally embeddable.
Thus $e_1$ is not connected to $v_3$, and similarly $e_2$ is not connected to $v_6$.
First assume that both $e_1$ and $e_2$ are not connected to $v_7$.
Since $G$ is simple, $e_1$ is connected to $v_1$, and $e_2$ is connected to $v_4$ as drawn in Figure~\ref{fig:v7} (e).
Then $G$ is isomorphic to $\mathbf{12}_2^7$ which is a minor minimal intrinsically chiral graph.
Now assume that $e_1$ is connected to $v_7$.
Then $e_2$ is connected to either $v_4$ or $v_7$ as drawn in Figure~\ref{fig:v7} (f).
In both cases, $G$ is isomorphic to $\mathbf{12}_1^7$ which is a minor minimal intrinsically chiral graph.

\begin{figure}[h!]
    \includegraphics{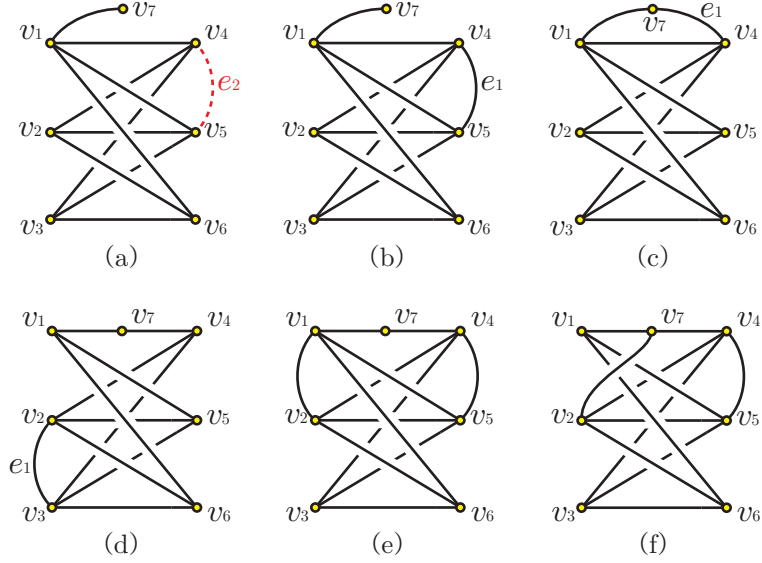}
    \caption{Case of $|G|=7$}
    \label{fig:v7}
\end{figure}

\subsection{$G$ has $K_5$ as a minor}\label{subsec:7v3}\

Now suppose that $G$ has $K_5$ as a minor.
Then there exist two edges $e_1$ and $e_2$ such that $K_5$ is obtained from $G$ by using two edge contractions at $e_1$ and $e_2$.
$$\xymatrix{
\ \ G \ \ \ar[r]^-{ \mathbb {C}} & \ \ G / e_1 \ \ \ar[r]^-{ \mathbb {C}} & \ \ \stackbelow{(G / e_1) / e_2}{K_5}\ \ \\
 }$$
We remark that if $e_1$(or $e_2$) connects two degree 3 vertices in $G/e_2$(or $G/e_1$), then $G/e_2$(or $G/e_1$) has $K_{3,3}$ as a minor.
This case has already been dealt with in the previous case that $G$ has $K_{3,3}$ as a minor.
So each $e_1$ and $e_2$ is connected at least one vertex with degree 1 or 2.
Let $v_1$ and $v_2$ be vertices with degree 1 or 2 in $G$ which are connected to $e_1$ and $e_2$, respectively.
Then remaining five vertices $v_3$, $v_4$, $v_5$, $v_6$ and $v_7$ have degree 4 or more.
If there are two vertices $v_3$ and $v_4$ which are not adjacent to $v_1$ and $v_2$, then both $v_3$ and $v_4$ are only adjacent to $v_5$, $v_6$ and $v_7$ in a planar graph $G\setminus \{v_3,v_4\}$.
Thus $G$ is achirally embeddable by Lemma~\ref{lem:ae}.
So without loss of generality, we assume that $v_1$ is adjacent to $v_3$ and $v_4$, and $v_2$ is adjacent to $v_5$ and $v_6$.
Then both $v_3$ and $v_4$ are only adjacent to $v_1$, $v_5$, $v_6$ and $v_7$ in a planar graph $G \setminus \{ v_3, v_4 \}$.
Thus $G$ is achirally embeddable by Lemma~\ref{lem:ae}.

\section{The order of $G$ is eight}\label{sec:8v}

In this section, we assume that $G$ is a non-planar graph of order eight and size twelve.
So $G$ has $K_{3,3}$ as a minor.
This implies that there exist three edges $e_1$, $e_2$ and $e_3$ such that $K_{3,3}$ is obtained from $G$ by using edge deletion at $e_1$, and two edge contractions at $e_2$ and $e_3$.
$$\xymatrix{
\ \ G \ \ \ar[r]^-{ \mathbb {D}} & \ \ G \setminus e_1 \ \ \ar[r]^-{ \mathbb {C}} & \ \ (G \setminus e_1) / e_2 \ \ \ar[r]^-{ \mathbb {C}} & \ \ \stackbelow{((G \setminus e_1) / e_2) /e_3}{K_{3,3}}\ \ \\
 }$$
Similar with Section~\ref{sec:7v}, we know that all four graphs in the process are simple.
Among two graphs $(G \setminus e_1) / e_2$ and $(G \setminus e_1) / e_3$, we have two cases that at least one of them is isomorphic to $H_1$, and all of them are isomorphic to $H_2$.
We prove the first case in Subsection~\ref{subsec:8v1} and the second case in Subsection~\ref{subsec:8v2}.  

\subsection{Case of $(G \setminus e_1) / e_2 \simeq H_1$}\label{subsec:8v1}\

Let $v_8$ be a vertex in $G \setminus e_1$ which is connected to $e_2$.
We first claim that $v_8$ is adjacent to a vertex in $\{ v_4, v_5, v_6 \}$ in $G\setminus e_1$.
Suppose for the contradiction that $v_8$ is not adjacent to any vertex in $\{ v_4, v_5, v_6 \}$.
Similar with the claim in Subsection~\ref{subsec:7v1}, there are two cases that $e_1$ is connected at most one vertex in $\{v_4,v_5,v_6\}$ or it connects two of them.
In the first case, we assume that $e_1$ is not connected to $v_4$ and $v_5$, and in the second case, we assume that $e_1$ connects $v_4$ and $v_5$.
Since both $v_4$ and $v_5$ are only adjacent to $v_1$, $v_2$ and $v_3$ in a planar graph $G \setminus \{ v_4, v_5 \}$, $G$ is achirally embeddable by Lemma~\ref{lem:ae}. 
Thus $v_8$ is adjacent to a vertex in $\{ v_4, v_5, v_6 \}$ in $G\setminus e_1$.
Without loss of generality, we assume that $v_8$ is adjacent to $v_4$ in $G \setminus e_1$.
If $v_8$ is not adjacent to any $v_2$ and $v_3$, then $G \setminus e_1$ has two mirror symmetrical embeddings $M(G \setminus e_1)^{v_2}_{v_3}$ and $M(G \setminus e_1)^{v_5}_{v_6}$.
So $G$ is achirally embeddable by Lemma~\ref{lem:c4}.
Thus we further assume that $v_8$ is adjacent to $v_2$ as drawn in Figure~\ref{fig:v80} (a).

Now consider the connection of $e_1$.
If $e_1$ connects $v_5$ and $v_6$ (or similarly $e_1$ is not connected to any vertex of them), then both $v_5$ and $v_6$ are only adjacent to $v_1$, $v_2$ and $v_3$ in a planar graph $G \setminus \{ v_5, v_6\}$.
So $G$ is achirally embeddable by Lemma~\ref{lem:ae}.
Thus we assume that $e_1$ is connected to $v_5$ and not $v_6$.
There are three vertices $v_4$, $v_7$ and $v_8$ which can be connected to $e_1$.
If $e_1$ is connected to $v_4$, then $G$ is isomorphic to $\mathbf{12}_1^8$ which is minor minimal intrinsically chiral as drawn in Figure~\ref{fig:v80} (b).
If $e_1$ is connected to $v_7$, it has a $\mathbf{11}_1^8$ as a minor.
If $e_1$ is connected to $v_8$, it has a mirror symmetry embedding as drawn in Figure~\ref{fig:v80} (c).

\begin{figure}[h!]
    \includegraphics{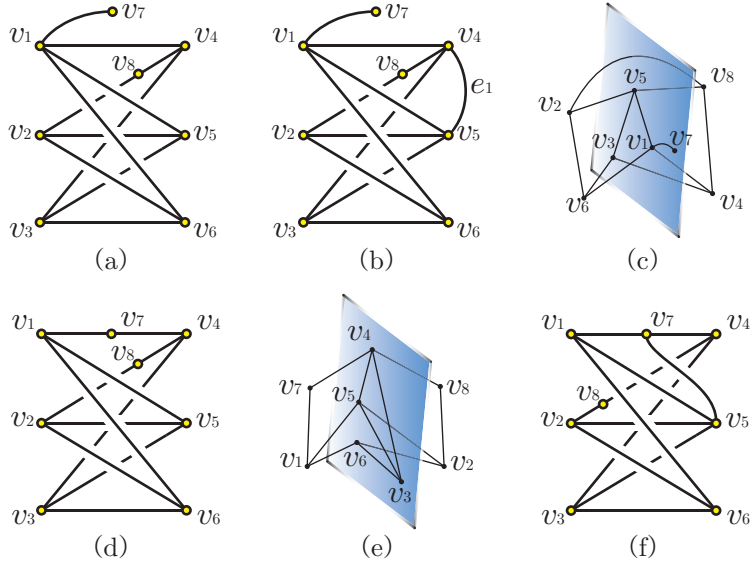}
    \caption{Case of $|G|=8$}
    \label{fig:v80}
\end{figure}

\subsection{Case of $(G \setminus e_1) / e_2 \simeq H_2$}\label{subsec:8v2}\

Let $v_8$ be a vertex in $G \setminus e_1$ which is connected to $e_2$.
Since $(G \setminus e_1) / e_2 \simeq (G \setminus e_1) / e_3 \simeq H_2$, $v_8$ is a degree 2 vertex.
If $v_8$ is not adjacent to any vertex in $\{v_2,v_3,v_5,v_6\}$, then $G \setminus e_1$ has two mirror symmetry embeddings $M(G \setminus e_1)^{v_2}_{v_3}$ and $M(G \setminus e_1)^{v_5}_{v_6}$.
So $G$ is achirally embeddable by Lemma~\ref{lem:c4}.
If $v_8$ is adjacent to two vertices in $\{v_2,v_3,v_5,v_6\}$, then $G$ has a $\mathbf{11}_1^8$ as a minor.
Without loss of generality, we assume that $v_8$ is adjacent to $v_2$ and $v_4$ as drawn in Figure~\ref{fig:v80} (d).
If $e_1$ connects $v_5$ and $v_6$ (or similarly $e_1$ is not connected to any vertex of them), then both $v_5$ and $v_6$ are only adjacent to $v_1$, $v_2$ and $v_3$ in a planar graph $G \setminus \{ v_5, v_6\}$.
So $G$ is achirally embeddable by Lemma~\ref{lem:ae}.
Thus we assume that $e_1$ is connected to $v_5$ and not $v_6$.
There are two vertices $v_4$ and $v_7$(or $v_8$) which can be connected to $e_1$.
If $e_1$ is connected to $v_4$ then it has a mirror symmetry embedding as drawn in Figure~\ref{fig:v80} (e).
If $e_1$ is connected to $v_7$ as drawn in Figure~\ref{fig:v80} (f), then $G$ is isomorphic to $\mathbf{12}_2^8$ which is minor minimal intrinsically chiral.

\section{The order of $G$ is nine}\label{sec:9v}

In this section, we assume that $G$ is a non-planar graph of order nine and size twelve.
So $G$ has $K_{3,3}$ as a minor.
This implies that there are three edges $e_1$, $e_2$ and $e_3$ such that $K_{3,3}$ is obtained from $G$ by using three edge contraction at $e_1$, $e_2$ and $e_3$.
$$\xymatrix{
\ \ G \ \ \ar[r]^-{ \mathbb {C}} & \ \ G / e_1 \ \ \ar[r]^-{ \mathbb {C}} & \ \ (G / e_1) / e_2 \ \ \ar[r]^-{ \mathbb {C}} & \ \ \stackbelow{((G / e_1) / e_2)/e_3}{K_{3,3}}\ \ \\
 }$$
Similar with Section~\ref{sec:7v} and \ref{sec:8v}, we know that all four graphs in the process are simple.
Among three graphs $(G / e_1) / e_2$, $(G / e_1) / e_3$ and $(G / e_2) / e_3$, we have two cases that at least one of them is isomorphic to $H_1$, and all of them are isomorphic to $H_2$.
We prove the first case in Subsection~\ref{subsec:9v1} and the second case in Subsection~\ref{subsec:9v2}.  

\subsection{Case of $(G / e_1) / e_2 \simeq H_1$}\label{subsec:9v1}\

Let $v_8$ and $v_9$ be vertices which is connected to $e_1$ and $e_2$ in $G$, respectively.
If $v_8$ or $v_9$ are adjacent to at most one vertex in $\{v_4,v_5,v_6\}$, then at least two vertices, say $v_4$ and $v_5$, in the set are only adjacent to $v_1$, $v_2$ and $v_3$ in a planar graph $G \setminus \{v_4,v_5\}$.
So $G$ is achirally embeddable by Lemma~\ref{lem:ae}.
Similarly, if both $v_8$ and $v_9$ are not adjacent to $\{v_2$ and $v_3\}$, then $G$ is achirally embeddable by Lemma~\ref{lem:ae}.
Without loss of generality, we assume that $v_8$ is adjacent to $v_2$ and $v_4$, and $v_9$ is adjacent to $v_5$.
Then we have two cases that $v_9$ is a degree 1 or 2 vertex in $G$.
If $v_9$ is a degree 1 vertex in $G$ as drawn in Figure~\ref{fig:v9} (a), then $G$ is isomorphic to $\mathbf{12}_1^9$ which is minor minimal intrinsically chiral.
So $v_9$ is a degree 2 vertex in $G$.
If $v_9$ is adjacent to $v_1$ or $v_3$, then $G$ contains $\mathbf{11}_1^8$ as a minor.
If $v_9$ is adjacent to $v_2$ then $G$ has a mirror symmetrical embedding as drawn in Figure~\ref{fig:v9} (b).

\begin{figure}[h!]
    \includegraphics{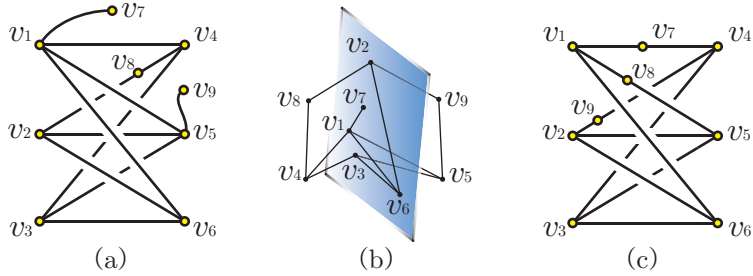}
    \caption{Case of $|G|=9$}
    \label{fig:v9}
\end{figure}

\subsection{Case of $(G / e_1) / e_2 \simeq H_2$}\label{subsec:9v2}\

Similar with the previous case, let $v_8$ and $v_9$ be vertices which is connected to $e_1$ and $e_2$ in $G$, respectively.
If $v_8$(or $v_9$) is not adjacent to any vertex in $v_1$ and $v_4$ in $G / e_2$(or $G / e_1$), then $G$ has $\mathbf{11}_1^8$ as a minor.
This implies that $v_8$ and $v_9$ are adjacent to at most one vertex in $\{v_2,v_3,v_5,v_6\}$.
Suppose that both $v_8$ and $v_9$ are not adjacent to any vertex in $v_2$ and $v_3$.
Then both $v_2$ and $v_3$ are only adjacent to $v_4$, $v_5$ and $v_6$ in a planar graph $G \setminus \{v_2,v_3\}$.
So $G$ is achirally embeddable by Lemma~\ref{lem:ae}.
Thus $v_8$ or $v_9$ is adjacent to at least one vertex in $\{v_2,v_3\}$.
Similarly, $v_8$ or $v_9$ is adjacent to at least one vertex in $\{v_5,v_6\}$.
Without loss of generality, we assume that $v_8$ is adjacent to $v_1$ and $v_5$, and $v_9$ is adjacent to $v_2$ and $v_4$ as drawn in Figure~\ref{fig:v9} (c).
Then $G$ has $\mathbf{11}_1^8$ as a minor.
Therefore, $G$ is not minor minimal intrinsically chiral in this case.

\section*{Acknowledgements}

The first author(Howon Choi) was supported by the Basic Science Research Program through the National Research Foundation of Korea (NRF) funded by the Ministry of Education (NRF-2022R1I1A1A01053856).
The corresponding author(Hyoungjun Kim) was supported by the National Research Foundation of Korea (NRF) grant funded by the Korea government Ministry of Science and ICT(NRF-2021R1C1C1012299).
The third author(Sungjong No) was supported by the National Research Foundation of Korea(NRF) grant funded by the Korea government Ministry of Science and ICT(NRF-2020R1G1A1A01101724).



\bibliography{mmic12.bib} 
\bibliographystyle{abbrv}

\end{document}